\documentclass{article}

\usepackage[utf8]{inputenc}
\usepackage{graphicx}
\usepackage{pifont}
\usepackage{amsmath}
\usepackage{multirow}
\usepackage{esvect}
\usepackage{breqn}
\usepackage[makeroom]{cancel}
\usepackage{pgf,tikz}
\usepackage{mathrsfs}
\usepackage{eufrak}
\usepackage{euscript}
\usetikzlibrary{arrows}
\usepackage{hyperref}
\usepackage[english]{babel}
\usepackage{amsthm}
\usepackage{bbding}
\usepackage{verbatim}
\usepackage{float}
\usepackage{algorithmic}
\usepackage{indentfirst}

\newtheorem{theorem}{Theorem}
\newtheorem{lemma}[theorem]{Lemma}
\newtheorem{case}{Case}

\newcommand{\showpic}[5]{\begin{figure}[#1]\centering\includegraphics[width=#2cm]{#3}\caption{#4}\label{#5}\end{figure}}

\title{Mathematics of a Sudo-Kurve}
\author{Tanya Khovanova \and Wayne Zhao}
\date{}

\begin{document}

\maketitle

\begin{abstract}
    We investigate a type of a Sudoku variant called Sudo-Kurve, which allows bent rows and columns, and develop a new, yet equivalent, variant we call a Sudo-Cube. We examine the total number of distinct solution grids for this type with or without symmetry. We study other mathematical aspects of this puzzle along with the minimum number of clues needed and the number of ways to place individual symbols.
\end{abstract}

\section{Introduction} 

Sudoku, a massively popular puzzle, was likely invented in 1979 by Howard Garnes. It hopped over to Japan, which gave it its modern name of ``Sudoku.'' After Wayne Gould got the first Sudoku published in Britain in 2004, the puzzle grew dramatically in popularity in 2005. As a result, Sudoku has been extensively studied \cite{RT}.

Sudoku puzzles are special cases of Latin squares, which were studied by Euler two centuries before. In 2005, extensive casework determined the total number of Sudoku solution grids to be $6670903752021072936960$ not accounting for symmetries \cite{totalnumber}; enough Sudoku puzzles to use up the informational, energy, and storage capacities of all of human history. Accounting for symmetries there are $5472730538$ solution grids. In 2013, the minimum number of clues needed to solve a given Sudoku was found. An extensive computer-assisted proof found the answer to be $17$ \cite{minnumber}.

``Regular'' Sudoku has thus been well-studied. There are now several dozen variants of Sudoku, appearing in many books, world puzzle championships, and websites \cite{mutantsudoku,beyondsudoku,sudokumasterpieces}. One of the websites at the forefront of this is the GMPuzzles blog \cite{gmpuzzles}. One particularly interesting variant found on GMPuzzles and elsewhere is that of the Sudo-Kurve, first invented by Steve Schaefer and named by Adam R. Wood \cite{gmpuzzles}. In Sudo-Kurve, each gray line connects the $9$ cells that comprise a row or column. These rows and columns are twisted into each other, but the rule about each row or column containing one of each number from $1$ to $9$ still holds.

In this paper we examine various mathematical aspects of a particular type of Sudo-Kurve we call a Cube Sudo-Kurve. We begin by examining the rules for Cube Sudo-Kurves and go slowly over an example of a puzzle in Section~\ref{section2}. Then, we find an equivalence between the Cube Sudo-Kurve and a Sudoku on a $3 \times 3 \times 3$ cube, and use this equivalence to deduce some strategies for solving Cube Sudo-Kurves in Section~\ref{section3}. We use these strategies in Section~\ref{section4} to determine that the total number of distinct solution grids without taking symmetry into account is 14515200. In Section~\ref{section5} we give an alternative calculation of the same number by using some symmetries. Later, in Section~\ref{section6}, we compute the total number of distinct solution grids accounting for all possible symmetries to be 2. In Section~\ref{section7} we consider Sudo-Cubes of other sizes. Then, in Section~\ref{section8}, we show that the minimum number of clues required to uniquely determine a solution of a Cube Sudo-Kurve is 8. Finally, we conclude with some observations on individual digit placement in Section~\ref{section9}.

\section{Cube Sudo-Kurve}
\label{section2}

In the paper, we consider one of the most interesting Sudo-Kurves that contain only three 3-by-3 squares. We call it a \textit{Cube Sudo-Kurve}.

The Cube Sudo-Kurve consists of three square blocks as in Figure~\ref{fig:emptycube}. The gray bent lines indicate how rows and columns continue. For example, the first row of the top left block becomes the last column of the middle block and continues to the first row of the bottom right block.

\begin{figure}[htb!]
\centering
\includegraphics[width=8cm]{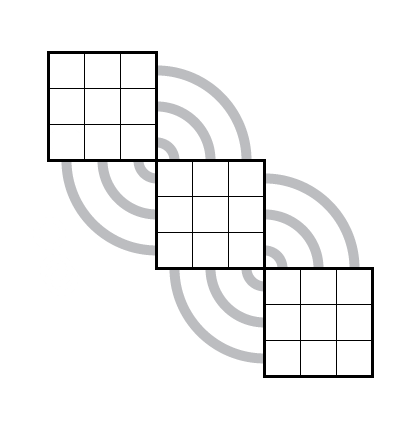}
\caption{Empty Cube Sudo-Kurve grid}\label{fig:emptycube}
\end{figure}

Here we provide a sample Cube Sudo-Kurve puzzle (Figure~\ref{samplepuzzle}) and solve it. This puzzle appeared on GMPuzzles on February 12, 2013 \cite{gmpuzzles}.

\showpic{htbp!}{6}{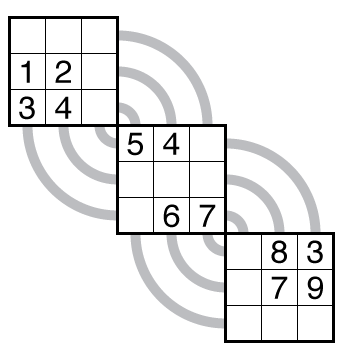}{A sample Cube Sudo-Kurve puzzle from \href{https://www.gmpuzzles.com/blog/2013/02/dr-sudoku-prescribes-38-sudo-kurve/}{the GMPuzzles blog}}{samplepuzzle}

Our first step is to notice that there are two instances of the symbol 3 and two of the symbol 7. Let us first consider the 3. The 3 in the block on the upper-left prevents the left column of the center block from containing a 3. Similarly, the 3 in the lower-right block prevents the right column of the center block from containing a 3. We therefore know the 3 must go in the center of the center block.

More generally, we note intuitively that given any two occurrences of a symbol we can determine the position of the third. This will be proven later. For now, this also means we can fill in the third occurrence of the 7, see Figure~\ref{3and7}.

\showpic{htbp!}{6}{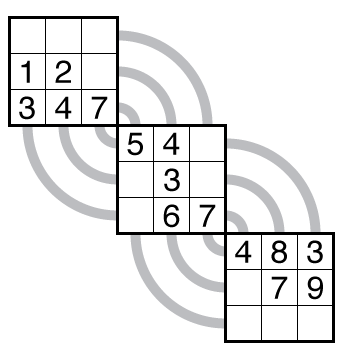}{Determining positions of 3 and 7}{3and7}

Now consider the middle row of the upper-left block, which is also the middle column of the center block and the middle row of the lower-right block. This row is missing the digits 8 and 5. However, we know from the 5 in the center block that the 5 cannot go into the upper-left block, so the 5 must belong in the lower-right block. This also fixes the value of 8 in the upper-left block, as in Figure~\ref{rowfill}.

\showpic{htbp!}{6}{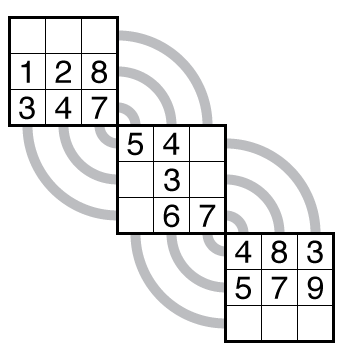}{Filling in a row}{rowfill}

Now we have two occurrences each of 5 and 8 so we can fill in their third occurrences (see Figure~\ref{5and8}).

\showpic{htbp!}{6}{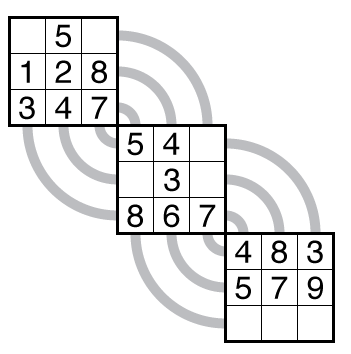}{Determining positions of 5 and 8}{5and8}

The column that starts from the left column of the upper-left block is missing a 2 and a 9. As there is a 9 in the lower-right block, the 9 in that column cannot be in that block, so it must be in the upper-left. We now can place the missing 2, as in Figure~\ref{columnfill}.

\showpic{htbp!}{6}{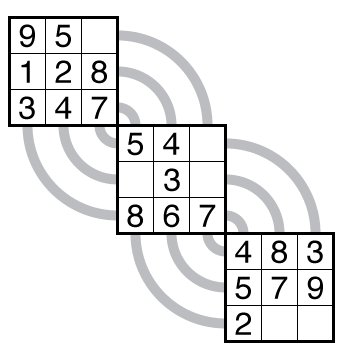}{Filling in a column}{columnfill}

Now we can fill in the remaining 2 and 9. Additionally, there is one digit missing in the upper-left block---a 6. We fill these in Figure~\ref{blockfill}.

\showpic{htbp!}{6}{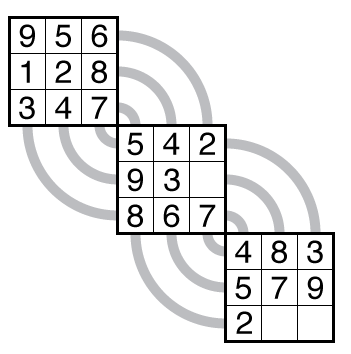}{Filling in a block}{blockfill}

We can now easily fill in the rest of the Cube Sudo-Kurve, see Figure~\ref{done!}.

\showpic{htbp!}{6}{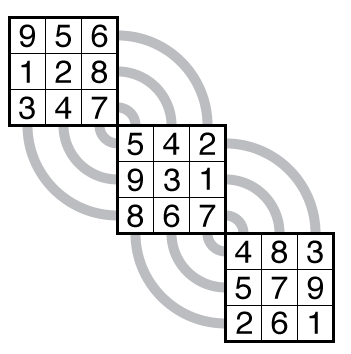}{Complete Cube Sudo-Kurve}{done!}

\section{Sudo-Cube}
\label{section3}

We now introduce another variation of Sudoku, which we later prove to be isomorphic to Cube Sudo-Kurve. We call it a \textit{Sudo-Cube}. This interpretation will make it easier to compute the total number of solution grids, with or without symmetry. The grid is a $3\times 3 \times 3$ cube. The digits $1$ though $9$ are placed in the cells of the cube so that the nine digits in each layer perpendicular to one of the axis are all distinct. To represent the Sudo-Cube in this paper we arrange  horizontal layers of the cube on a plane next to each other as in Figure~\ref{fig:emptysudocube}. We can assume that the bottom layer is on the left and the top layer on the right.

We denote the left block which represents the bottom layer of the cube as $B1$, the center block representing the middle layer of the cube as $B2$, and the right block representing the top layer of the cube as $B3$.

\begin{figure}[htb!]
\centering
\includegraphics[width=8cm]{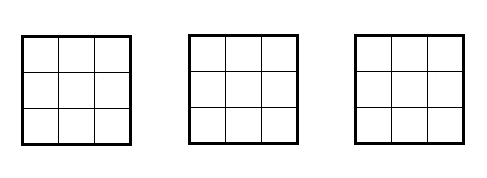}
\caption{Empty Sudo-Cube grid}\label{fig:emptysudocube}
\end{figure}

Now we prove that these two grids are equivalent.

\begin{theorem}
A filled Cube Sudo-Kurve is isomorphic to a filled Sudo-Cube.
\end{theorem}

\begin{proof}
We match the top left block of the Cube Sudo-Kurve to $B1$---the bottom layer of the Cube Sudo-Cube. 
We match the bottom right square of the Cube Sudo-Kurve, to the top layer of the Sudo-Cube; that is, to $B3$. We flip the middle square of the Cube Sudo-Kurve with respect to the anti-diagonal and match the result to the middle layer of the Sudo-Cube, aka $B2$.

We can see that the bent row corresponding to the first row of the top left block of the Cube Sudo-Kurve, together with the last column of the middle block and the first row of the bottom right block of the Cube Sudo-Kurve, becomes the first row of each of the blocks of the Sudo-Cube. In other words, the rows of the Cube Sudo-Kurve correspond to front-facing squares in the Sudo-Cube. Similarly, the columns of the Cube Sudo-Kurve correspond to the side-facing squares.
\end{proof}

Note that the bent rows and columns now correspond to actual rows and corresponding columns of all blocks in a Sudo-Cube. 

We now denote each square in the Sudo-Cube with a coordinate. The square in the $k$th row in the $m$th column in the $n$th block shall be denoted with the triple $(k,m,n)$.

Visualizing this Sudo-Kurve as a cube helps us prove other results. 

\begin{lemma}
Knowing two instances of a symbol uniquely defines the location of the third instance of the symbol.
\end{lemma}

\begin{proof}
Indeed, assume two instances of a symbol are found at $(k_1, m_1, n_1)$ and $(k_2, m_2, n_2)$. We try to find all triples $(k_3, m_3, n_3)$ where the third instance of the symbol can be located.
Now since, there are only three values of the first coordinate possible, knowing $k_1$ and $k_2$ leaves us with $k_3$. Similarly, we can determine $m_3$ and $n_3$, so there is only one possible position for the third instance of the symbol to be found.
\end{proof}

\section{Number of solution grids}
\label{section4}

We start this section by discussing the situations where we cannot finish placing a particular symbol. Suppose we fill blocks $B1$ and $B2$ with digits. Then each digit has a unique place in $B3$ it has to go to. But it could happen that two different digits need to be in the same cell in $B3$. We call this situation an \textit{obstruction}. An obstruction only happens if the two digits use the same two rows and the same two columns in $B1$ and $B2$. 

We have two cases here. The first case is that the digits $a$ and $b$ swap places when moving from $B1$ to $B2$. We call this a \textit{swap}. The second case is when the digits $a$ and $b$ form opposite corners of a rectangle after projection. We call this a \textit{cross}.

We can now count the number of ways to complete this Cube Sudo-Kurve grid. 

\begin{theorem}
The number of ways to fill the Cube Sudo-Kurve grid is $14515200$.
\end{theorem}

\begin{proof}
Note that the number of ways to fill in $B2$ and $B3$ does not depend on which of the $9!$ ways we fill in $B1$. We therefore assume the grid is filled in like in Figure~\ref{fig:standardstart}.
\showpic{ht}{10}{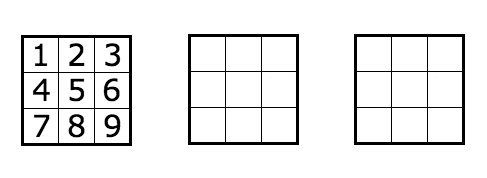}{Standard Starting Position}{fig:standardstart}
Later will multiply by $9!$ to get the total number of solutions to this Sudo-Kurve.

We now have $3$ cases to contend with, as described below.

\begin{case}
The first row of $B2$ is comprised entirely of the elements of either the second or the third row of $B1$.
\end{case}

Let us assume the first row of $B2$ is comprised entirely of the elements of the second row of $B1$, that is $\{4,5,6\}$ in some order. This forces the second row of $B2$ to be $\{7,8,9\}$ in some order, and forces the third row of $B2$ to be $\{1,2,3\}$ in some order.

We see that the $4$ can really only go in $2$ spots: the second or third box in the first row of $B2$. Similarly, the $5$ can only go in $2$ spots, as can the $6$, and it is easy to see that overall, this gives $2$ ways to order the first row (this is equivalent to the fact that there are $2$ derangements of $3$ objects), which also locks in the positions of the $4$, 5, and 6 in $B3$. Similarly, there are $2$ ways to arrange the elements of the second row of $B2$ and $2$ ways to arrange the third, which gives a total of $2^3 = 8$. Similarly, there are 8 ways to fill the grid when the first row of $B2$ is $\{7,8,9\}$ making the total of 16.

\begin{case}
The first row of $B2$ contains three elements from different columns and not all from the same row.
\end{case}

Assume that the first row of $B2$ contains $4$ and $5$. Then the third element of the first row of $B2$ would need to be among $\{7,8,9\}$. But since $7$ is in the same column as $4$, and $8$ is in the same column as $5$, the third element of the first row of $B2$ is $9$. All other ways of choosing digits for the first row of $B2$ are equivalent to this one as we can reshuffle rows and columns. 

It can be seen that there are $2 \cdot \binom{3}{2} = 6$ ways to choose the elements for the top row. Namely, there are $2$ ways to choose which row we will have two elements from and $\binom{3}{2}$ ways to pick the two elements from that row, which determines the third number.

We now assume the top row is $\{4,5,9\}$ in some order. We use braces to indicate that the order is not known yet. We can then deduce the following about what elements are in what rows as seen in Figure~\ref{fig:deductions}.

\showpic{ht}{10}{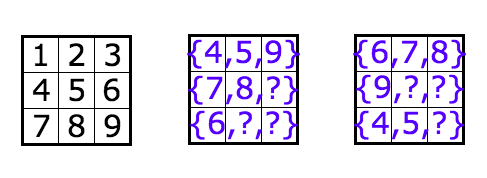}{Deductions}{fig:deductions}

We know that $7$ and $8$ must be in the second and $6$ in the third row of $B2$. 

As it turns out, once we have the information about what elements are in the first row, there are only two ways to finish the puzzle.

To prove this, we simply solve the puzzle. Assume the first row is $5,9,4$ in that order. We can then immediately fill in more cells as in  Figure~\ref{fig:s2.1}.

\showpic{ht}{10}{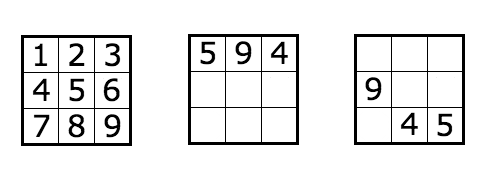}{Subcase 1 of Case 2}{fig:s2.1}

But in fact, we can uniquely determine the rest of the grid. Note that in $B2$, the bottom-right square cannot be filled with a $6$, $7$, or $8$. It therefore contains instead a $1$, $2$, or $3$. It cannot be 3, as we already have a $3$ in the right column. It cannot be a 2 as then $2$ is swapped with $9$. It has to be $1$. After that, $2$ has to go in the first column. It cannot be in the middle row, as then 2 needs to be in the bottom right corner of $B3$ which is occupied. That means 2 has to be in the bottom left corner. Therefore, number 3 is in the middle row in $B2$. We know $3$ cannot be in the first column as it clashes (forms a swap) with $4$. Therefore, the only way is as shown in Figure~\ref{fig:s2.1deductions}.

\showpic{ht}{10}{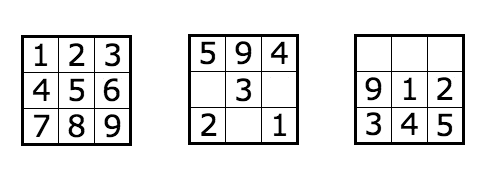}{Subcase 1 of Case 2 Deductions}{fig:s2.1deductions}

After that the rest is uniquely defined.

Thus, for each of the $6 \cdot 2 = 12$ ways to choose and fix the first row, we only get one solution, which implies there are actually $12$ total solutions in this case.

\begin{case}
The first row of $B2$ contains exactly two elements of the same row of $B1$. Similarly, there are exactly two elements from the same column.
\end{case}

This case is similar to Case $2$, except we have $\{4,5,7\}$ in the first row, for example, instead of $\{4,5,9\}$. Note that $4$ and $5$ are elements of the second row of $B1$, while $7$ is an element of the third row of $B1$, and happens to also be in the same column as $4$ in $B1$. In total, there are $2 \cdot \binom{3}{2} \cdot 2 = 12$ ways to choose the elements in the first row of $B2$. If we assume the first row is $\{4,5,7\}$, we can determine the following information about what rows contain what numbers as in Figure~\ref{fig:c3}.

\showpic{ht}{10}{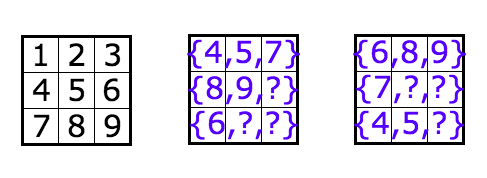}{Case 3}{fig:c3}

However, we can actually determine more information. It turns out the order of the numbers in the first row is fixed. Note that $4$ and $7$ must go in the last two cells of the first row of $B1$. Integer $5$ must therefore be in the first cell of the first row of $B1$, and $5$ must be in the bottom-right-most box of $B3$. But now notice $4$ cannot be the second element in the first row of $B2$, or else the $4$ must occupy the same position as the $5$ in $B3$. The position of $4$, and subsequently of the $7$, is therefore fixed.

There is now enough information in this problem to completely solve this puzzle now, see Figure~\ref{fig:c3solved}.

\showpic{ht}{10}{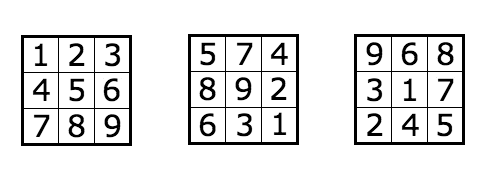}{Case 3 Solved}{fig:c3solved}

All other ways of choosing digits for the first row of $B2$ are equivalent to this one as we can reshuffle rows and columns. Since there is only one solution for each initial arrangement, there are 12 solution grids in this case.

Therefore the total number of solution grids with $B1$ fixed is $16+12+12 = 40$, implying there are $40 \cdot 9! = 14515200$ total solution grids.
\end{proof}

\section{Alternative count}
\label{section5}

We can also exploit symmetries to determine the total number of solution grids. The symmetries of the Sudo-Cube are similar to those of regular Sudoku. They are, swapping two parallel layers of the Sudo-Cube, rotating or reflecting the entire cube, and relabeling the digits of the cube. Note that swapping layers manifests in the display of the Sudo-Cube as either like swapping blocks, rows in the blocks, or columns in the blocks.

Now assume, without loss of generality, that we fix the number $5$ in the center of $B1$. Then the number $5$ in a solution to this Cube Sudo-Kurve must be in row $1$ or row $3$ and column $1$ or column $3$ of $B2$. We can switch these columns and rows arbitrarily, until the $5$ is in the top left corner of $B2$. We then relabel the rest of the cells so that $B1$ is in standard configuration. 

Thus, after counting the number of solutions after fixing $B1$ and the top left corner of $B2$ as 5, we can multiply it by $4 \cdot 9!$. There happen to be $10$ ways, and we get the same number of solution grids as before.

\subsection{Explicit Cases}

Let us call each of these 10 ways \textit{sudo-cases}. We can list out all $10$ sudo-cases with the $5$ in the upper-left of the middle square. The first four correspond to Case 1, when the first row of $B2$ is comprised entirely of the elements of either the second or the third row of $B1$. Together with the condition that 5 is in the upper left corner of $B2$ we get that the second row of $B2$ must be 5, 6, and 4. Here are sudo-cases 1 through 4.

\begin{figure}[H]
    \centering\includegraphics[width=7cm,trim={0 1cm 0 0},clip]{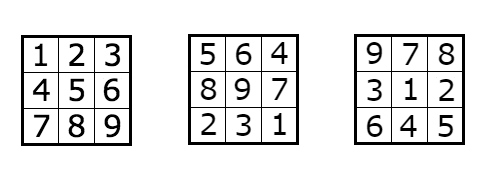}
    \centering\includegraphics[width=7cm,trim={0 1cm 0 0},clip]{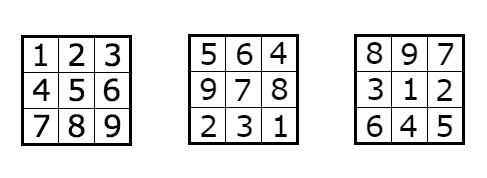}
    \centering\includegraphics[width=7cm,trim={0 1cm 0 0},clip]{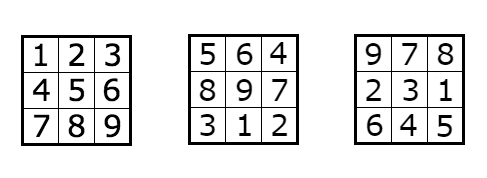}
    \centering\includegraphics[width=7cm,trim={0 1cm 0 0},clip]{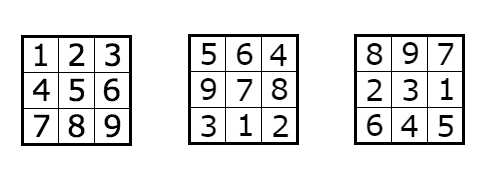}
\end{figure}

Now we consider the second case when the first row of $B2$ contains three elements from different columns, not all from the same row. There are three possibilities for the first row of $B2$. Sudo-cases 5, 6, and 7 are presented below.

\begin{figure}[H]
    \centering\includegraphics[width=7cm,trim={0 1cm 0 0},clip]{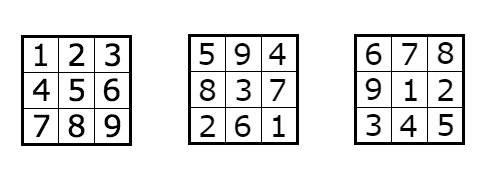}
    \centering\includegraphics[width=7cm,trim={0 1cm 0 0},clip]{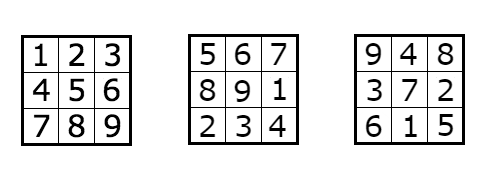}
    \centering\includegraphics[width=7cm,trim={0 1cm 0 0},clip]{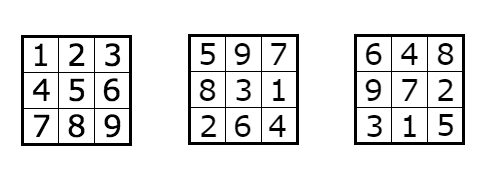}
\end{figure}

In the Case 3, the first row of $B2$ contains exactly two elements of the same row of $B1$ and exactly two elements from the same column. There are three sudo-cases 8, 9, and 10 below.

\begin{figure}[H]
    \centering\includegraphics[width=7cm,trim={0 1cm 0 0},clip]{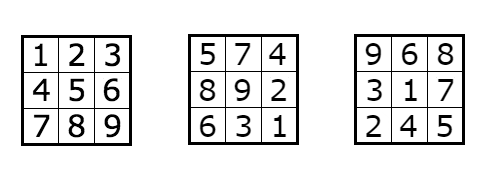}
    \centering\includegraphics[width=7cm,trim={0 1cm 0 0},clip]{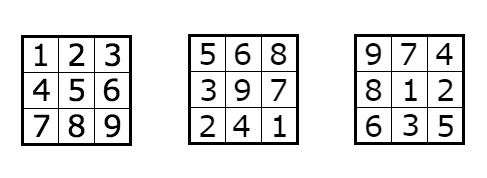}
    \centering\includegraphics[width=7cm,trim={0 1cm 0 0},clip]{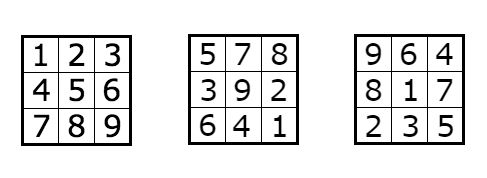}
\end{figure}

\section{The number of different solution grids up to symmetries}
\label{section6}

We want to calculate the number of different solution grids up to symmetries. The symmetries are: 
\begin{itemize}
    \item Relabeling of the digits.
    \item Swapping different layers of the cube. In other words, swapping blocks, rows, and columns.
    \item Movements of the cube: rotations and reflections.
\end{itemize}

We already have 10 sudo-cases. But we did not account for all possible symmetries. We only accounted for relabeling and two types of swaps: swapping top and bottom rows and swapping the first and the last column. The last two swaps are equivalent to reflections of the cube.

We did not account for permuting blocks $B1$, $B2$, and $B3$. We also did not yet account for swapping the left and right layers with the middle layer. We did not yet account for swapping the front and back layers with the middle layer. We also did not try every move of the cube.

Many of these ten sudo-cases are really equivalent. We start by showing that sudo-case 1 is not equivalent to any other sudo-case.

\begin{lemma}
Sudo-case 1 will remain sudo-case 1 under all transformations.
\end{lemma}

\begin{proof}
Suppose we pick a direction on the cube that is up. Then the blocks $B1$, $B2$, and $B3$ are uniquely defined. Consider a row of 9 digits that spans all three blocks. Each row is a partition of digits 1-9 into three sets of three elements. Sudo-case 1 is the only sudo-case where for all initial direction this partition is the same for every row. 

Relabeling, shuffling layers, and movements of the cube do not change this property. Thus sudo-case 1 remains equivalent to sudo-case 1 under all transformations.
\end{proof}

Now we continue with our main result of classifying Cube Sudo-Kurve solution grids up to symmetry.

\begin{theorem}
There are only two distinct sudo-cases under symmetry: sudo-case 1 and sudo-case 2. Sudo-cases 3 through 10 are equivalent to sudo-case 2.
\end{theorem}

\begin{proof} 
We show below that using a reflection with respect to the main diagonal in each of the blocks $B1$, $B2$, and $B3$, sudo-cases 5, 6, 7, and 9 are equivalent to sudo-cases 2, 3, 4, and 8 correspondingly. By swapping $B1$ and $B2$, sudo-case 4 becomes sudo-case 2. By swapping $B2$ and $B3$, sudo-case 3 becomes sudo-case 2. By swapping the top and the middle layer in each block, sudo-case 10 becomes sudo-case 8. By rotation sudo-case 8 becomes sudo-case 2. 

Sudo-cases 2-9 are thus equivalent.
\end{proof} 

\subsection{Reflection with respect to the main diagonal in each block}

To start, we consider a symmetry that keeps the 5 in place: the reflection of each block with respect to the main diagonal. We then need to relabel our digits so that the first block is in the standard form. That is, we swap the digits in the following pairs: (2,4), (3,7), and (6,8).

\textbf{Sudo-case 2}, which is
\begin{figure}[H]
    \centering\includegraphics[width=7cm,trim={0 1cm 0 0},clip]{SK5_1_2.png}
\end{figure}
\noindent becomes, upon reflection,
\begin{figure}[H]
    \centering\includegraphics[width=7cm,trim={0 1cm 0 0},clip]{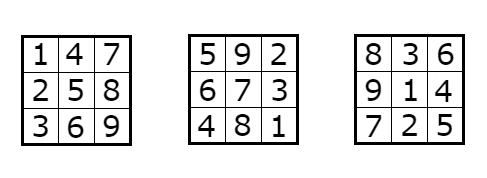}
\end{figure}
\noindent which itself becomes, upon relabeling,
\begin{figure}[H]
    \centering\includegraphics[width=7cm,trim={0 1cm 0 0},clip]{SK5_1_5.png}
\end{figure}
\noindent which is sudo-case 5.

\textbf{Sudo-case 3}, or,
\begin{figure}[H]
    \centering\includegraphics[width=7cm,trim={0 1cm 0 0},clip]{SK5_1_3.png}
\end{figure}
\noindent becomes, upon reflection,
\begin{figure}[H]
    \centering\includegraphics[width=7cm,trim={0 1cm 0 0},clip]{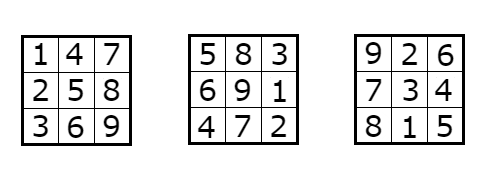}
\end{figure}
\noindent which, upon relabeling, is
\begin{figure}[H]
    \centering\includegraphics[width=7cm,trim={0 1cm 0 0},clip]{SK5_1_6.png}
\end{figure}
\noindent or sudo-case 6.

\textbf{Sudo-case 4} is
\begin{figure}[H]
    \centering\includegraphics[width=7cm,trim={0 1cm 0 0},clip]{SK5_1_4.png}
\end{figure}
\noindent which becomes, upon reflection,
\begin{figure}[H]
    \centering\includegraphics[width=7cm,trim={0 1cm 0 0},clip]{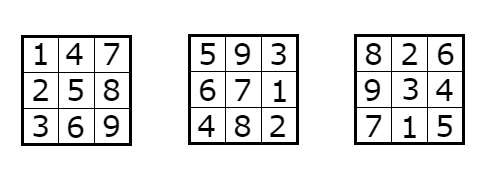}
\end{figure}
\noindent and upon relabeling becomes
\begin{figure}[H]
    \centering\includegraphics[width=7cm,trim={0 1cm 0 0},clip]{SK5_1_7.png}
\end{figure}
\noindent which is sudo-case 7.

Finally, \textbf{Sudo-case 8} is
\begin{figure}[H]
    \centering\includegraphics[width=7cm,trim={0 1cm 0 0},clip]{SK5_1_8.png}
\end{figure}
\noindent which becomes, upon reflection,
\begin{figure}[H]
    \centering\includegraphics[width=7cm,trim={0 1cm 0 0},clip]{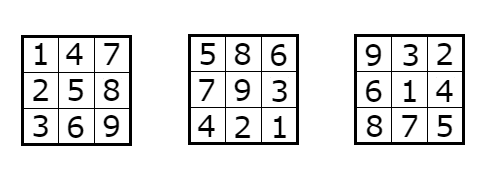}
\end{figure}
\noindent which, upon relabeling, becomes
\begin{figure}[H]
    \centering\includegraphics[width=7cm,trim={0 1cm 0 0},clip]{SK5_1_9.png}
\end{figure}
\noindent which is sudo-case 9.

\subsection{Swap B1 and B2}

We now try another operation: switching $B1$ and $B2$. With this we show that sudo-case 2 is equivalent to sudo-case 4.

\textbf{Sudo-case 2} is
\begin{figure}[H]
    \centering\includegraphics[width=7cm,trim={0 1cm 0 0},clip]{SK5_1_2.png}.
\end{figure}
Now we swap $B1$ and $B2$ and get
\begin{figure}[H]
    \centering\includegraphics[width=7cm,trim={0 1cm 0 0},clip]{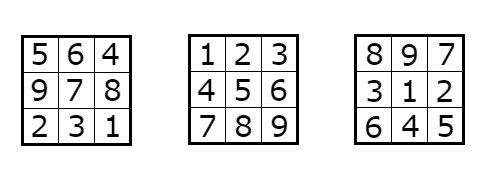}.
\end{figure}
Now, if we were to relabel this configuration, the 7's would become 5's. Unfortunately, this would mean the 5 would not be in the upper-left corner of $B2$. To fix this, we swap the first and last row to get
\begin{figure}[H]
    \centering\includegraphics[width=7cm,trim={0 1cm 0 0},clip]{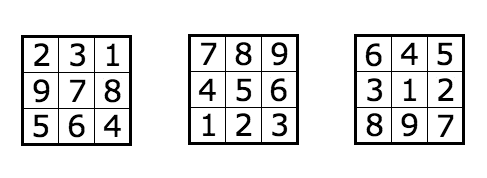}.
\end{figure}
After relabeling we get
\begin{figure}[H]
    \centering\includegraphics[width=7cm,trim={0 1cm 0 0},clip]{SK5_1_4.png}
\end{figure}
\noindent which is sudo-case 4.

\subsection{Swap B2 and B3}

We now swap $B2$ and $B3$. This allows us to show that sudo-case 2 and sudo-case 3 are the same.

\textbf{Sudo-case 2} is
\begin{figure}[H]
    \centering\includegraphics[width=7cm,trim={0 1cm 0 0},clip]{SK5_1_2.png}
\end{figure}
\noindent and after swapping $B2$ and $B3$ we get
\begin{figure}[H]
    \centering\includegraphics[width=7cm,trim={0 1cm 0 0},clip]{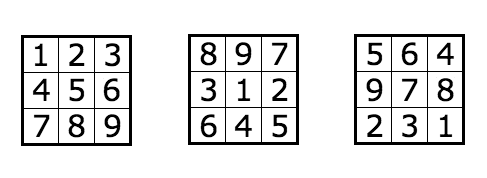}.
\end{figure}
In order to get the 5 from the bottom-right of $B2$ to the top-left, we need to swap the bottom and top rows, then swap the left and right columns:
\begin{figure}[H]
    \centering\includegraphics[width=7cm,trim={0 1cm 0 0},clip]{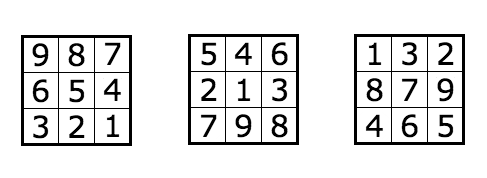}.
\end{figure}
And relabeling gives
\begin{figure}[H]
    \centering\includegraphics[width=7cm,trim={0 1cm 0 0},clip]{SK5_1_3.png}
\end{figure}
\noindent which is sudo-case 3.

\subsection{Swap top and middle row in each block}

We now swap the top and middle rows of each block. This will show that sudo-cases 8 and 10 are equivalent.

\textbf{Sudo-case 8} is
\begin{figure}[H]
    \centering\includegraphics[width=7cm,trim={0 1cm 0 0},clip]{SK5_1_8.png}
\end{figure}
\noindent and swapping the top and middle rows gives us
\begin{figure}[H]
    \centering\includegraphics[width=7cm,trim={0 1cm 0 0},clip]{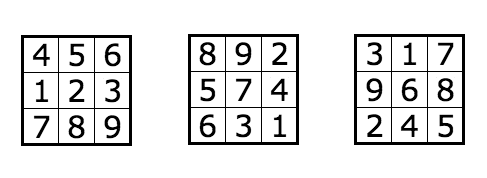}.
\end{figure}
We now need the 2 in $B2$ to be in the upper-left, so we swap the left- and rightmost columns of each block.
\begin{figure}[H]
    \centering\includegraphics[width=7cm,trim={0 1cm 0 0},clip]{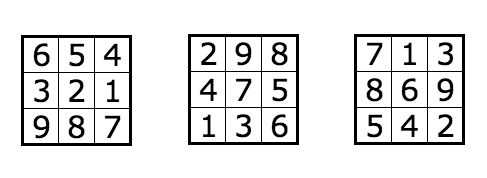}.
\end{figure}
Finally, relabeling gives us
\begin{figure}[H]
    \centering\includegraphics[width=7cm,trim={0 1cm 0 0},clip]{SK5_1_10.png}
\end{figure}
\noindent which is sudo-case 10.

\subsection{Rotation}

Now we start with sudo-case 8
\begin{figure}[H]
    \centering\includegraphics[width=7cm,trim={0 1cm 0 0},clip]{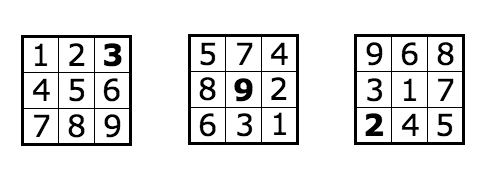}.
\end{figure}
We rotate it along the 3-9-2 space diagonal (in bold above and below):
\begin{figure}[H]
    \centering\includegraphics[width=7cm,trim={0 1cm 0 0},clip]{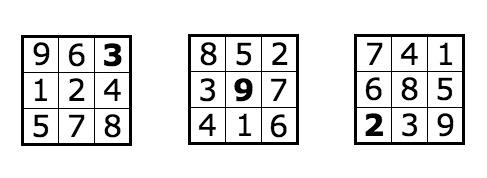}.
\end{figure}
We then need a 2 in the upper-left of $B2$, so we swap the left and right columns:
\begin{figure}[H]
    \centering\includegraphics[width=7cm,trim={0 1cm 0 0},clip]{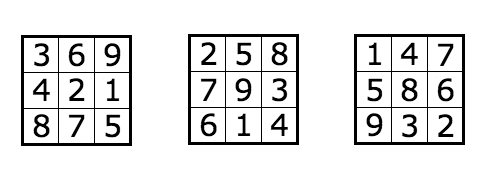}.
\end{figure}
And finally, relabeling gives
\begin{figure}[H]
    \centering\includegraphics[width=7cm,trim={0 1cm 0 0},clip]{SK5_1_7.png}
\end{figure}
\noindent which is sudo-case 7. Sudo-case 8 is therefore equivalent to sudo-case 2.

\section{Other Cube Sizes}
\label{section7}

It is interesting to determine similar properties for cubes of smaller sizes. 

For a $1\times 1 \times 1$ cube, there is clearly only $1$ way to fill the grid whether or not we consider symmetries.

For a $2 \times 2 \times 2$ cube, there are $24$ distinct ways to arrange the numbers in the first layer. After that, the only way to arrange the top layer is to put everything in the diametrically opposite place. The total number of solutions grids without considering symmetry for a cube of size 2 is 24.

Since all solution grids of this cube are derived from relabeling the numbers in the first layer, all solution grids of the $2 \times 2 \times 2$ cube are isomorphic to each other under relabeling.

Thus, the sequence of the number of different solution grids as a function of the grid size starts as 1, 24, 14515200. The sequence of the number of distinct solution grids under symmetry starts as 1, 1, 2.

\section{The minimum number of clues}
\label{section8}

It is known that in ``regular'' Sudoku the minimum number of clues needed to uniquely determine a solution grid is 17 \cite{minnumber}.

We can now determine the minimum number of clues required to force a unique solution in a Cube Sudo-Kurve. For a Cube Sudo-Kurve of size 1, we do not need any symbols. For a Cube Sudo-Kurve of size 2, we must need at least three symbols so we can differentiate between all symbols. Here is an example of such a minimal Cube Sudo-Kurve:

\begin{figure}[H]
    \centering\includegraphics[width=3.5cm,trim={0 0 0 0},clip]{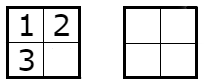}.
\end{figure}

For a Cube Sudo-Kurve of size 3, we must have at least eight different symbols. Otherwise, we would not be able to tell the difference between the two or more missing symbols. To prove that $8$ is the required minimum, we created the following two puzzle grids in Figure~\ref{fig:minpuzzle1} and Figure~\ref{fig:minpuzzle2}, which have only one solution. We encourage readers to solve them both.

\showpic{H}{7}{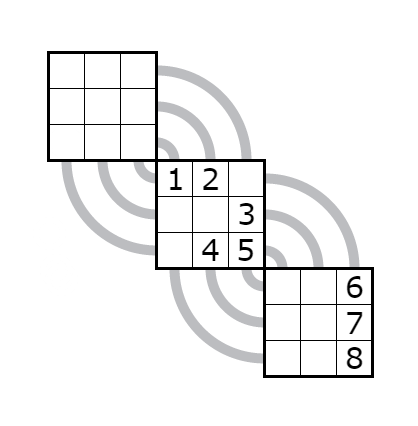}{Puzzle with Minimum Number of Clues (easier)}{fig:minpuzzle1}
\showpic{H}{7}{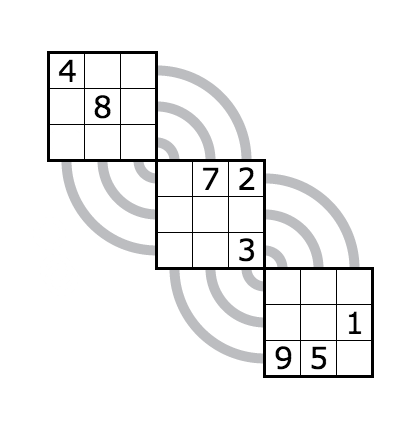}{Puzzle with Minimum Number of Clues (harder)}{fig:minpuzzle2}

\section{Placing single digits}
\label{section9}

We now count the number of ways to place a particular digit in a cube of size $n$.

\begin{lemma}
For $n \times n \times n$ cube, the number of ways to place any given digit is $(n!)^2$.
\end{lemma}

\begin{proof}
There are $n^2$ locations to place the symbol in the bottom block. In the next layers from the bottom, one column and one row is forbidden. Therefore, we have leftover spaces isomorphic to $(n-1)$ cube, so we can inductively compute that there are $((n-1)!)^2$ ways to place the symbol in the rest of the cube, and the final total is $n^2 \cdot ((n-1)!)^2 = (n!)^2$.
\end{proof}

In particular, for $n = 2$ we get $4$, and for $n=3$, we get $36$.

We can also look at the placement of the same digit up to isomorphisms of the cube. That is we are looking at the shape that is formed inside a cube by the same symbol. Looking at shapes allows us to ignore the actual digits, that is we are studying the shape interaction up to relabeling. The idea to use shapes of symbols was used by Conway and Ryba to describe different Latin squares of size 4 up to movements of the plane and relabeling \cite{CR}.

For $n=2$, the only way we can place a given digit is to place it along a main diagonal of the cube. 

For $n=3$, if our digit occupies the center, then it has to use up one of the four main diagonals. Up to rotations, all of these shapes are the same.

Suppose a digit takes up a corner and does not use the center. Without loss of generality, we can say that the corner has coordinates $(1,1,1)$. Then the other two points must have coordinates $(2,3,3)$, $(3,2,2)$, up to movements of the cube. In other words, one symbol is at a vertex of a cube, another symbol is in the center of a face not adjacent to the vertex, and the third symbol is in the middle of an edge that is neither adjacent to the vertex nor to the face from which the center is used.

Thus, for each corner, there are three possibilities. There are total of 24 arrangements of a single digit in this case. Up to rotations and reflections all these shapes are the same, and they all form a scalene triangle.

If there is no symbol in a corner, then all three of them must be in the middles of edges no two of which share a vertex. For example the digits could be at $(1,2,1)$, $(2,3,3)$, and $(3,1,2)$.

There are eight cases like this. All such shapes are isomorphic to an equilateral triangle. 

A given symbol in a 3 by 3 by 3 cube could only be in one of the three shapes described above. Exactly one digit, the one in the center, corresponds to the diagonal, six digits have to use a triangle passing through a corner, and two digits form equilateral triangles.

Let us consider sudo-case 1. We can recognize it by the shapes formed by each of the symbols. For instance, the digit 7 forms a diagonal. We consider planes that go through this diagonal and two opposite edges. There are 3 such planes and they are listed below.

\begin{itemize}
    \item The plane that is formed by the main diagonals of each block contains only digits 1, 5, and 9.
    \item The plane formed by the last row in $B1$, the middle row in $B2$ and the top row in $B3$ contains only digits 7, 8, and 9.
    \item The plane formed by the last column in $B1$, the middle column in $B2$, and the first column in $B3$ contains only digits 3, 6, and 9.
\end{itemize}

By checking all sudo-cases, we see that the only case where the three planes that include two opposite edges and the main diagonal formed by the same digit all have exactly three different symbols is sudo-case 1. This property is invariant under relabeling and movements of the cube. Given that the we get to the standard form of sudo-case 1 by relabeling and reflections of the cube, that means this property stays before relabeling and reflections. That means we can recognize the sudo-case 1 by this property before any action.

\section{Acknowledgments}

We would like to thank MIT's PRIMES-USA for making this research paper possible.

\end{document}